\documentclass[10pt,a4paper,reqno,centertags]{article}
\usepackage{a4wide,amsmath,amsthm,amssymb}
\usepackage[T1]{fontenc}

\newtheorem{thm}{Theorem}[section]

\newtheorem{lemma}[thm]{Lemma}
\newtheorem{prop}[thm]{Proposition}

\theoremstyle{remark}
\newtheorem{rem}[thm]{Remark}
\newtheorem{ex}[thm]{Example}

\newenvironment{example}{\begin{ex}\rm}{\qee\end{ex}}
\newcommand{\qee}{\mbox{\hspace{0.2mm}}\hfill$\triangle$}

{\nonumber}
\newcommand{\di}{{\mathrm d}}
\newcommand{\R}{{\mathbb R}}
\newcommand{\C}{{\mathbb C}}
\newcommand{\ii}{{\mathrm i}}
\newcommand{\intt}{\frac{1}{2\pi}\int_{0}^{2\pi}}
\newcommand{\del}{\partial}

\begin{document}
\begin{center}
 {\LARGE\bf Integrable systems and holomorphic curves}\\[15pt]
 {\sc Paolo Rossi}\\
 {(Ecole Polytechnique, Paris and MSRI, Berkeley)}
\end{center}

\vspace{1cm}

\begin{abstract}
In this paper we attempt a self-contained approach to infinite dimensional Hamiltonian systems appearing from holomorphic curve counting in Gromov-Witten theory. It consists of two parts. The first one is basically a survey of Dubrovin's approach to bihamiltonian tau-symmetric systems and their relation with Frobenius manifolds. We will mainly focus on the dispersionless case, with just some hints on Dubrovin's reconstruction of the dispersive tail. The second part deals with the relation of such systems to rational Gromov-Witten and Symplectic Field Theory. We will use Symplectic Field theory of $S^1\times M$ as a language for the Gromov-Witten theory of a closed symplectic manifold $M$. Such language is more natural from the integrable systems viewpoint. We will show how the integrable system arising from Symplectic Field Theory of $S^1\times M$ coincides with the one associated to the Frobenius structure of the quantum cohomology of $M$.
\end{abstract}

\section{A model for infinite-dimensional Hamiltonian systems}

In this section we introduce the language and formalism necessary to study the kind of infinite dimensional Hamiltonian systems which are related to Gromov-Witten theory. We will define a formal (infinite dimensional) phase space with a Poisson structure and consider a space of functions over this phase space which, together with the Poisson structure, will give rise to a Hamiltonian dynamics described in the form of a system of PDEs with one spatial and one temporal variable. This material comes almost entirely from \cite{DZ}.

\vspace{0.5cm}

{\bf Phase space and local functionals.} We will consider as our phase space a space of formal maps $\mathcal{L}=\mathcal{L}(\R^n)=\{u:S^1\to \mathbb{R}^n\}$ from the circle to an $n$ dimensional real affine space. We will denote by $x$ the coordinate on $S^1$ and by $u^\alpha=u^\alpha(x)$ the $\alpha$-th component of the map $u$, with $\alpha=1,\ldots,n$. Formal here refers to the choice of the space of functions over $\mathcal{L}$ that we are going to make (indeed the maps $u$ and the space $\mathcal{L}$ itself should be seen as convenient notational tools to describe a theory which is completely defined at the level of the space of functions, see also Remark \ref{fourier}). As functions on $\mathcal{L}$ we will consider the space of local functionals of the form
$$F[u]=\frac{1}{2\pi}\int_0^{2\pi} f(x,u,u_x,u_{xx},\ldots) \di x$$
where $u_x$ denotes the derivative of $u$ with respect to $x$, $u_{xx}$ the second derivative, and so on and all of them are considered as independent formal variables. We will also denote by $u^{(k)}$ the $k$-th $x$-derivative of $u$ and by $u^{\alpha,k}(x)$ its $\alpha$-th component. The density $f(x,u,u_x,u_{xx},\ldots)$ is a differential polynomial, i.e. a smooth function in the $x$ and $u$ variables (on $S^1\times \R^n$) and a polynomial of any order in the higher derivatives. Usually we will denote a local functional and its density by higher and lower case versions of the same letter, respectively. Every time $f(x,u,u_x,u_{xx},\ldots)$ coincides with a total $x$-derivative via the chain rule $$\partial_x g(x,u,u_x,u_{xx},\ldots):=\frac{\partial g}{\partial x}+\sum_{\alpha,k} \frac{\partial g}{\partial u^{\alpha,k}} u^{\alpha,k+1}$$
then $$F[u]=\frac{1}{2\pi}\int_0^{2\pi} f(x,u,u_x,u_{xx},\ldots) \di x=\frac{1}{2\pi}\int_0^{2\pi} \partial_x g(x,u,u_x,u_{xx},\ldots) \di x=0$$
This is of course equivalent to considering the space of local functionals as the space of differential polynomials $f(x,u,u_x,u_{xx},\ldots)$ up to total $x$-derivatives.

\vspace{0.5cm}

{\bf Poisson brackets.} We want to endow our formal loop space $\mathcal{L}$ with a Poisson structure. The Poisson brackets between two local functionals will have the form
$$\left\{F[u],G[u]\right\}=\frac{1}{4\pi^2}\int_0^{2\pi}\int_0^{2\pi} \frac{\delta F}{\delta u^\alpha(x)} \left\{u^\alpha(x),u^\beta(y)\right\} \frac{\delta G}{\delta u^\beta(y)}\, \di x \, \di y$$
where
$$\frac{\delta F}{\delta u^\alpha(x)}:=\sum_s (-1)^s\partial_x^s \frac{\partial f}{\partial u^{\alpha,s}}$$
$$\left\{u^\alpha(x),u^\beta(y)\right\}:= \sum_{s=0}^\infty a_s^{\alpha \beta}(u,u_x,u_{xx},\ldots) \delta^{(s)}(x-y)$$
and where $\delta^{(s)}(x-y)=\partial_x^s \delta(x-y)$  is a formal delta-distribution defined by $$\frac{1}{2\pi}\int_0^{2\pi} f(x,u,u_x,u_{xx},\ldots) \delta^{(s)}(x-y) \, \di x:= \partial_x^s f(x,u,u_x,u_{xx},\ldots) |_{x=y}$$
We then see that, after the action of the delta-function (which elides one of the integrals), the Poisson bracket of two local functionals is still a local functional. Of course in order to have a genuine Poisson bracket the usual antisymmetry, Leibnitz and Jacobi identities must hold, which imposes some constraints on the form of the coefficients $a_s^{\alpha \beta}(u,u_x,u_{xx},\ldots)$. We will study below some of such constraints.

\begin{rem}\label{fourier}
Notice that a possible choice of coordinates on the space $\mathcal{L}$ is given by the formal Fourier components of the formal functions $u^\alpha(x)$ on the circle. The functions $u^\alpha(x)$ and their Fourier components are formal in the sense that it is coherent with our choice of formal local functionals that we don't require any sort of convergence for the Fourier series of $u^\alpha(x)$. More precisely, besides the infinite set of formal variables $u,u_x,u_{xx},\ldots$, we consider another infinite set of formal variables $\{u^\alpha_0,p^\alpha_k,q^\alpha_k\}$ which we want to see as Fourier coefficients in a formal Fourier series:
$$u^\alpha(x)=u^\alpha_0+\sum_{k=1}^\infty \left(p^\alpha_k \mathrm{e}^{\mathrm{i}kx}+q^\alpha_k \mathrm{e}^{-\mathrm{i}kx}\right)$$
The above formula and its formal $x$-derivatives (term by term in the Fourier series) should be seen as a change of coordinates between $u,u_x,u_{xx},\ldots$ and the $\{u^\alpha_0,p^\alpha_k,q^\alpha_k\}$ variables.
A local functional $F[u]$ can then be written as an infinite polynomial in the $\{u^\alpha_0,p^\alpha_k,q^\alpha_k\}$ variables (i.e. a finite degree infinite sum of monomials in $\{u^\alpha_0,p^\alpha_k,q^\alpha_k\}$) just by substituting $u,u_x,u_{xx},\ldots$ with the corresponding Fourier series and then formally integrating $f(x,u,u_x,u_{xx},\ldots)$ term by term in $x$.
The definition (or inverse transformation) of $u^\alpha_0$, $p^\alpha_k$ and $q^\alpha_k$ in terms of the formal map $u$ is given as local functionals by
$$u^\alpha_0=\frac{1}{2\pi}\int_0^{2\pi} u^\alpha \,\di x$$
$$p^\alpha_k=\frac{1}{2\pi}\int_0^{2\pi} u^\alpha \mathrm{e}^{-\mathrm{i}kx} \,\di x$$
$$q^\alpha_k=\frac{1}{2\pi}\int_0^{2\pi} u^\alpha \mathrm{e}^{+\mathrm{i}kx} \,\di x$$
Coherently the delta-function has the Fourier series
$$\delta(x-y)=\sum_{k=-\infty}^{+\infty}\mathrm{e}^{\mathrm{i}k(x-y)}$$
from which one can compute the Poisson bracket $\left\{p^\alpha_k,q^\beta_l\right\}$.
\end{rem}

\vspace{0.5cm}

{\bf Vector fields and Hamiltonian vector fields.} The vector fields on our formal loop space will have the general form
$$a=a^0(x,u,u_x,u_{xx},\ldots)\frac{\partial}{\partial x} + \sum_{\beta,k}a^{\beta,k}(x,u,u_x,u_{xx},\ldots)\frac{\partial }{\partial u^{\beta,k}}$$
and their action as differential operators on local functionals is of course
$$a(F):=\frac{1}{2\pi}\int_0^{2\pi} a(f) \di x$$
We will be interested in particular in \emph{evolutionary vector fields}, which are characterized by the vanishing of the $\frac{\partial}{\partial x}$ component and the property $[\partial_x,a]=0$. Notice that of course, in particular, $\partial_x=\frac{\partial}{\partial x}+u^{\alpha,k+1}\frac{\partial}{\partial u^{\alpha,k}}$ has this last property. It is straightforward to check that from the definition of evolutionary vector field it follows that its general form is
$$a=\sum_{\beta,k} \partial_x^k a^\beta \frac{\partial}{\partial u^{\beta,k}}$$
for some differential polynomials $a^1(x,u,u_x,u_{xx},\ldots),\ldots,a^n(x,u,u_x,u_{xx},\ldots)$. In case such differential polynomials do not depend explicitly on $x$, the evolutionary vector field will be called \emph{translation invariant}.
As usual the presence of a Poisson bracket and a selected Hamiltonian function $H$ (functional $H[u]$, in our case) generates a dynamic on the phase space in the form of a Hamiltonian vector field. On the formal loop space $\mathcal{L}$ the Hamiltonian flow of $H[u]$ has infinitesimal expression (formal Hamiltonian vector field, or Hamilton's equations)
$$u_t^\alpha=\left\{u^\alpha(x),H[u]\right\}=\frac{1}{2\pi}\int_0^{2\pi} \left\{u^\alpha(x),u^\beta(y)\right\}\frac{\delta H}{\delta u^\beta(y)} \, \di y= U^\alpha(x,u,u_x,u_{xx},\ldots)$$
More precisely the above components $U^\alpha(u,u_x,u_{xx},\ldots)$ are the $n$ differential polynomials defining the \emph{evolutionary} Hamiltonian vector field of $H[u]$. This, for $\alpha=1,\ldots,n$, can be seen as a system of evolutionary PDEs for the functions $u^\alpha=u^\alpha(x,t)$ representing the Hamiltonian flow on $\mathcal{L}$. We will call a system of this type an infinte-dimensional Hamiltonian system. When the Hamiltonian does not depend on $x$ explicitly, its Hamiltonian vector field is translation invariant.\\

\begin{example}
The most celebrated example of infinite-dimensional Hamiltonian system (because of all its beautiful properties and its appearence in many different contexts of mathematics) is Korteweg-deVries (KdV for short) equation. In our language it corresponds to the choice of the formal loop space $\mathcal{L}=\mathcal{L}(\R)$ (so $u$ has only one component $u(x)$) and the Poisson structure $\{u(x),u(y)\}=\delta '(x-y)$. The Hamiltonian functional generating the KdV flow is $H[u]=\frac{1}{2\pi}\int_0^{2\pi} \left(\frac{u^3}{6}+\frac{1}{24}(u^2_x+2u_{xx})\right)\, \di x$. Hence the KdV equation (Hamilton's equation for $H[u]$) is
$$u_t=u u_x+\frac{1}{12}u_{xxx}$$
\end{example}

\vspace{0.5cm}

{\bf Long wave limit and extended formal loop space.} In order to have an extra formal variable keeping track of the number of $x$-derivatives entering our formulas it is costumary to rescale both the $x$ and $t$ variables by
$$x\mapsto \frac{x}{\epsilon} \hspace{2cm} t\mapsto \frac{t}{\epsilon}$$
The limit $\epsilon\to 0$ is called the long wave limit, in reference to its effect on the behavior of solutions to the Hamiltonian PDEs. Notice that the definition of the delta-function implies the corresponding rescaling $\delta(x,y)\to\frac{1}{\epsilon} \delta(x-y)$. The general local functional hence becomes a polynomial in $\epsilon$. However we will, from now on, extend our class of local functionals including \emph{formal power series} in $\epsilon$ of the form
$$F[u]=\frac{1}{2\pi}\int_0^{2\pi} \left(\sum_{k=0}^\infty f_k(x,u,\ldots,u^{(k)}) \epsilon^k \right)\, \di x$$
where, assigning degree $i$ to the variable $u^{\alpha,i}$, $f_k$ has degree $k$.

The general Poisson bracket takes the form
\begin{equation}
\begin{split}
\left\{u^\alpha(x),u^\beta(y)\right\}=&\frac{1}{\epsilon} h^{\alpha\beta}(u) \delta(x-y)\\
&+ g^{\alpha \beta}(u) \delta '(x-y) + \Gamma^{\alpha \beta}_{\gamma}(u) u^\gamma_x \delta(x-y)\\
&+ O(\epsilon)
\end{split}
\end{equation}

\vspace{0.5cm}

{\bf Poisson brackets of hydrodynamic type.} Poisson brackets for which the $\frac{1}{\epsilon}$-term vanishes are called \emph{hydrodynamic}. They were studied by Dubrovin and Novikov and their main property is
\begin{thm}[\cite{DN}]
Let $$\left\{u^\alpha(x),u^\beta(y)\right\}=g^{\alpha \beta}(u) \delta '(x-y) + \Gamma^{\alpha \beta}_{\gamma}(u) u^\gamma_x \delta(x-y)+O(\epsilon)$$ be a Poisson bracket of hydrodynamic type. Then (from the antisymmetry, Jacobi and Liebniz properties):
\begin{itemize}
\item[(i)] $g^{\alpha\beta}$ is symmetric
\item[(ii)] if $g^{\alpha\beta}$ is nondegenerate, then it is a flat metric and $\Gamma^{\alpha \beta}_{\gamma}(u)=-g^{\alpha\mu}\Gamma_{\mu \gamma}^{\beta}(u)$ are its Christoffel symbols.
\end{itemize}
\end{thm}

We will assume from now on that the matrix $g^{\alpha\beta}$ is indeed nondegenerate. This result implies in particular that, if we choose flat coordinates for such metric, $g^{\alpha\beta}(u)$ is mapped to a constant matrix $\eta^{\alpha\beta}$ and the term involving Christoffels vanishes.

In fact it turns out that the full tail of positive order in $\epsilon$ in the Poisson bracket can be also killed if we allow more general changes of our formal variables in the form of a \emph{Miura transformation}
$$u^\alpha \mapsto \tilde{u}^\alpha:=\sum_{k=0}^\infty \tilde{u}^\alpha_k (u,u_x,,\ldots,u^{(k)})\epsilon^k$$
where $\det(\frac{\partial\tilde{u}^\alpha_0(u)}{\partial u^\beta})\neq 0$ and assigning degree $i$ to the variable $u^{\alpha,i}$, $\tilde{u}^\alpha_k(u,u_x,\ldots,u^{(k)})$ is a polynomial of degree $k$. Notice that such transformations indeed form a group, where inversion consists in solving an infinite order ODE for $u^\alpha$ recursively (and completely algebraically) in power series of $\epsilon$ and in terms of the known functions $\tilde{u}^\beta$ and their $x$-derivatives. This in fact is our main motivation for having introduced the long wave limit variable $\epsilon$ and the extended formal loop space. We have in particular
\begin{thm}[\cite{DZ}]
Given a Poisson bracket of hydrodynamic type, there exist a Miura transformation $$u^\alpha \mapsto \tilde{u}^\alpha:=\sum_{k=0}^\infty \tilde{u}^\alpha_k (u,u_x,,\ldots,u^{(k)})\epsilon^k$$ such that, in the new coordinates, the Poisson bracket is reduced to the normal form
$$\left\{u^\alpha(x),u^\beta(y)\right\}=\eta^{\alpha \beta} \delta '(x-y)$$
The local functionals $\frac{1}{2\pi}\int_0^{2\pi}\tilde{u}^\alpha(u,u_x,u_{xx},\ldots;\epsilon)\,\di x$ are Casimirs of the Poisson bracket (i.e. they commute with any other local functional). In particular $\tilde{u}^\alpha_0(u)$ are flat coordinates for the metric $g^{\alpha \beta}(u)$.
\end{thm}
Denoting by $\{\tilde{u}^\alpha_0,\tilde{p}^\alpha_k,\tilde{q}^\alpha_k\}$ the formal Fourier components of $\tilde{u}^\alpha$ we see that their only nonzero Poisson bracket is
$$\{\tilde{p}^\alpha_k,\tilde{q}^\beta_l\}=\mathrm{i}\,k\, \eta^{\alpha\beta}\,\delta_{k,l}$$

\vspace{0.5cm}

\begin{example}
There is another Poisson structure and Hamiltonian giving rise to the KdV equation on $\mathcal{L}(\R)$. Namely,
$$\{u(x),u(y)\}=u \delta'(x-y)+\frac{1}{2} u_x \delta(x-y) + \frac{\epsilon^2}{8}\delta'''(x-y)$$
$$H[u]=\intt \frac{u^2}{2} \di x$$
The Miura transformation reducing such Poisson structure to normal form is found by solving recursively in $\epsilon$ the Riccati equation
$$u=\frac{\tilde{u}^2}{4}-\frac{\epsilon}{2\sqrt{2}}\tilde{u}_x$$
\end{example}

\vspace{0.5cm}

{\bf Bihamiltonian structures of hydrodynamic type.} Following Magri \cite{Ma}, we say that two Poisson structures of hydrodynamic type $\{\cdot,\cdot\}_1$ and $\{\cdot,\cdot\}_2$ form a \emph{bihamiltonian structure} if and only if the linear combination $\{\cdot,\cdot\}_\lambda:=\{\cdot,\cdot\}_2-\lambda\{\cdot,\cdot\}_1$ is still a Poisson structure of hydrodynamic type. The following theorem gives a way to find commuting Hamiltonians for a bihamiltonian structure.

\begin{thm}[\cite{Ma}]\label{recursion}
Let $\{u^\alpha(x),u^\beta(y)\}_\lambda$ be reduced to the normal form $-\lambda \eta^{\alpha,\beta} \delta'(x-y)$ by the Miura transformation
$$u^\alpha \mapsto \tilde{u}^\alpha(u,u_x,u_{xx},\ldots;\epsilon,\lambda)=\sum_{p=-1}^{\infty} \frac{1}{\lambda^{p+1}}f_p^\alpha(u,u_x,u_{xx},\ldots;\epsilon).$$
Then
$$\{\cdot,F_{p+1}^\alpha\}_1=\{\cdot,F_p^\alpha\}_2$$
and, in particular,
$$\{F_p^\alpha,F_q^\beta\}_1=\{F_p^\alpha,F_q^\beta\}_2=0$$
where $F_p^\alpha[u]=\intt f_p^\alpha(u,u_x,u_{xx},\ldots;\epsilon)\di x$.
\end{thm}

\vspace{0.5cm}

By definition, the new variables $\tilde{u}^\alpha$ are densities of Casimirs of the Poisson bracket $\{u^\alpha(x),u^\beta(y)\}_\lambda$.

\vspace{0.5cm}

\begin{example}
The Kdv Poisson structure
$$\{u(x),u(y)\}_1=\delta'(x-y)$$
$$\{u(x),u(y)\}_2=u \delta'(x-y)+\frac{1}{2}u_x \delta(x-y)+\frac{\epsilon^2}{8}\delta'''(x-y)$$
form a bihamiltonian structure. $\{u(x),u(y)\}$ is reduced to normal form by solving the Riccati equation
$$\frac{\ii\epsilon}{2\sqrt{2\lambda}}\tilde{u}_x - \frac{\tilde{u}^2}{4\lambda}=u-\lambda$$
recursively in $\epsilon$ and $\lambda$ as
$$\tilde{u}=\sum_{m=0}^\infty \frac{\tilde{u}_m}{(\sqrt{\lambda})^m}+2\lambda$$
We find in this way that $\tilde{u}_m$ is a total $x$-derivative when $m$ is odd and we define
$$F_p=\intt \tilde{u}_{2p+2} \di x$$
Such Hamiltonians are symmetries of the KdV equation.
\end{example}

Notice that, although $F_p^\alpha$ form an infinite number of commuting Hamiltonians, they may very well fail to be linearly independent. The following example is instructive.

\vspace{0.5cm}

\begin{example}\label{Toda}
On the loop space $\mathcal{L}(\R^2)$ with coordinates $u(x),v(x)$ consider the bihamiltonian structure
$$\{u(x),u(y)\}_1=\{v(x),v(y)\}_1=0$$
$$\{u(x),v(y)\}_1=\frac{1}{\epsilon}(\delta(x-y+\epsilon)-\delta(x-y))$$
and
$$\{u(x),u(y)\}_2=\frac{1}{\epsilon}(e^{v(x+\epsilon)}\delta(x-y+\epsilon)-e^{v(x)}\delta(x-y+\epsilon))$$
$$\{v(x),v(y)\}_2=\frac{1}{\epsilon}(\delta(x-y+\epsilon)-\delta(x-y-\epsilon))$$
$$\{u(x),v(y)\}_2=\frac{1}{\epsilon}u(x)(\delta(x-y+\epsilon)-\delta(x-y))$$
Applying the first one to the Hamiltonian $$H[u,v]=\intt \left(\frac{1}{2}u^2+e^v\right) \di x$$ one obtains the Toda equations
$$u_t=\frac{1}{\epsilon}(e^{v(x+\epsilon)}-e^{v(x)})=(e^{v})_x+\frac{\epsilon}{2} (e^v)_{xx}+\ldots$$
$$v_t=\frac{1}{\epsilon}(u(x)-u(x-\epsilon))=v_x-\frac{\epsilon}{2}v_{xx}+\ldots$$
We can apply the above theorem to find mutually commuting Hamiltonians for the Toda bihamiltonian structure (incidentally containing the Toda Hamiltonian $H[u,v]$ itself). Let us study the Casimirs for the first Poisson bracket. Its reducing transformation is given by
$$v=\tilde{v}$$
$$u=\frac{1}{\epsilon}(\tilde{u}(x+\epsilon)-\tilde{u}(x))$$
whose inverse is
$$\tilde{v}=v$$
$$\tilde{u}=\epsilon\left(\frac{\del_x}{e^{\epsilon \del_x}-1}\right)u=\sum_{n=0}^\infty \frac{B_n}{n!}(\epsilon \del_x)^n u$$
where $B_n$ are the Bernoulli numbers. Hence, the Casimirs of $\{\cdot,\cdot\}_1$ are $\intt v\, \di x$ and $\intt \tilde{u}\, \di x=\intt u\, \di x$.

Notice now that $\intt v\, \di x$ is a Casimir of $\{\cdot,\cdot\}_2$ too! This implies that the one component of the reducing transformation for $\{\cdot,\cdot\}_\lambda=\{\cdot,\cdot\}_2-\lambda\{\cdot,\cdot\}_1$ is trivial
$$\tilde{v}(u,v,u_x,v_x,\ldots;\epsilon,\lambda)=v$$
so that one of the two infinite sequences of Hamiltonians $F^\alpha_p[u,v]$ vanishes.
\end{example}

\vspace{0.5cm}

{\bf Regular Poisson pencils.} We plan now to give a definition of a bihamiltonian system which is sufficiently complete, meaning that its symmetries span a sufficiently large space of local functionals. To this end, starting from a bihamiltonian structure of hydrodynamic type on the loop space $\mathcal{L}(\R^n)$ and its associated bihamiltonian system $F_p^\alpha[u]$, we define its \emph{space of conserved quantities} as
$$\mathcal{I}=\left\{h(u,u_x,u_{xx},\ldots;\epsilon) \;|\; \{H[u],F^\alpha_p[u]\}_1=0 \right\}/\C$$
$$\tilde{\mathcal{I}}=\mathcal{I}/\mathrm{Im} \del_x = \left\{H[u]=\intt h(u,u_x,u_{xx},\ldots;\epsilon) \, \di x \;|\; h \in \mathcal{I} \right\}$$

\vspace{0.5cm}

These are in general infinite dimensional spaces so, in order to control their size, we will use a filtration  which is compatible with the bihamiltonian structure. Namely a \emph{compatible flag of conserved quantities} for the bihamiltonian system $F^\alpha_p[u]$ is a flag
$$\mathcal{F}_{-1}\mathcal{I} \subset \mathcal{F}_0\mathcal{I} \subset\ldots\subset \mathcal{I}$$
inducing a flag
$$\mathcal{F}_{-1}\tilde{\mathcal{I}} \subset \mathcal{F}_0\tilde{\mathcal{I}} \subset\ldots\subset \tilde{\mathcal{I}}$$
such that
$$\dim \mathcal{F}_{-1}\mathcal{I}=\dim \mathcal{F}_{-1}\tilde{\mathcal{I}}=\dim \frac{\mathcal{F}_{k}\mathcal{I}}{\mathcal{F}_{k-1}\mathcal{I}}=\dim \frac{\mathcal{F}_{k}\tilde{\mathcal{I}}}{\mathcal{F}_{k-1}\tilde{\mathcal{I}}}=n \hspace{0.8cm} k\geq 0$$
and such that
$$\{\cdot,\mathcal{F}_{-1}\tilde{\mathcal{I}}\}\;=\;0$$
$$\{\cdot,\mathcal{F}_{k-1}\tilde{\mathcal{I}}\}_2\;\subset\;\{\cdot,\mathcal{F}_k\tilde{\mathcal{I}}\}_1 \hspace{0.8cm} k\geq 0$$
$$\{\mathcal{F}_i\tilde{\mathcal{I}},\mathcal{F}_j\tilde{\mathcal{I}}\}_1=0 \hspace{0.8cm} i,j\geq -1$$

\vspace{0.5cm}

A bihamiltonian structure of hydrodynamic type possessing a compatible flag of conserved quantities is called \emph{regular}. A bihamiltonian structure of hydrodynamic type whose associated system $F^\alpha_p[u]$ spans a compatible flag is called \emph{nonresonant}. For instance, as we saw, the bihamiltonian structure of example \ref{Toda} is resonant. However we will see how it is regular anyway.

\vspace{0.5cm}

{\bf Tau-symmetry.} A regular bihamiltonian structure is said to possess a \emph{$\tau$-symmetry} (or to be \emph{$\tau$-symmetric}) if and only if there exist a basis of the corresponding flag of conserved quantities
$$\langle h_{1,-1},\ldots,h_{n,-1}\rangle \;=\; \mathcal{F}_{-1}\mathcal{I}$$
$$\langle h_{1,p},\ldots,h_{n,p} \rangle \;=\; \frac{\mathcal{F}_p\mathcal{I}}{\mathcal{F}_{p-1}\mathcal{I}} \hspace{0.8cm} p\geq 0$$
such that
$$\{h_{\alpha,p-1},H_{\beta,q}\}_1\;=\;\{h_{\beta,q-1},H_{\alpha,p}\}_1\;=:\; \del_x \Omega_{\alpha,p;\beta,q} $$
Here, as usual, the Poisson bracket of a local functional $H_{\beta,q}[u]$ and a density $h_{\alpha,p_1}(u,u_x,u_{xx},\ldots;\epsilon)$ is a differential polynomial. In our specific case such polynomial has to be a total $x$-derivative, since $H_{\beta,p}[u]$ and $H_{\alpha,p}[u]$ Poisson-commute for any $\alpha,p,\beta,q$. Hence the definition of the symmetric matrix $\Omega_{\alpha,p;\beta,q}(u,u_x,u_{xx},\ldots;\epsilon)$.

We will include in our definition of $\tau$-symmetric basis the requirement that $h_{1,0}(u,u_x,u_{xx},\ldots;\epsilon)$ is the generator of spacial translations, i.e.
$$\{u^\alpha,H_{1,0}\}_1\;=\;u^\alpha_x$$
This implies that any solution to the associated Hamiltonian system of PDEs will be a function $u^\alpha=u^\alpha(x+t^{1,0},t;\epsilon)$ so that we can indeed identify the $x$ and $t^{1,0}$ variables, using the notation $u^\alpha=u^\alpha(t;\epsilon)$

From now on we will make the Miura transformation to normal coordinates for the first Poisson bracket
$$u^\alpha \mapsto \eta^{\alpha\beta}h_{\beta,-1}(u,u_x,u_{xx},\ldots;\epsilon)$$
For simplicity we will still denote by $u^\alpha$ such normal coordinates.

The dynamical system generated by the densities $h_{\alpha,p}$ is called a \emph{bihamiltonian $\tau$-symmetric (BHTS) hierarchy}.

\vspace{0.5cm}

The main property of a BHTS hierarchy is given by the following theorem.
\begin{thm}[\cite{DZ}]
For any solution $u^\alpha=u^\alpha(t;\epsilon)$ of a BHTS hierarchy, there exists a function $\tau=\tau(t;\epsilon)$ such that
$$\Omega_{\alpha,p;\beta,q}(u(t;\epsilon),u_x(t;\epsilon),\ldots;\epsilon)=\epsilon^2 \frac{\del \log \tau}{\del t^{\alpha,p}\del t^{\beta,q}}$$
In particular, in normal coordinates,
$$u^\alpha(t;\epsilon)\;=\; \epsilon^2 \frac{\del \log \tau}{\del x \del t^{\alpha,0}}$$
\end{thm}

\vspace{0.5cm}

{\bf Frobenius manifolds.} An important result in Dubrovin's work (see e.g. \cite{DZ}) is the classification of the $\epsilon=0$ limit of BHTS hierarchies in terms of Frobenius manifolds. The abstract definition of Frobenius manifold can be found easily in the literature (\cite{D1}, \cite{DZ}), but we prefer to stick to a more operational definition, as in \cite{D}. A \emph{Frobenius manifold} structure on an open subset $U\subset\R^n$ is a (smooth or analytic) solution $F:U\to \R$ of the \emph{Witten-Dijkgraaf-Verlinde-Verlinde (WDVV)} equations
$$\frac{\del^3 F}{\del u^\alpha \del u^\beta \del u^\gamma} \eta^{\gamma \delta} \frac{\del^3 F}{\del u^\delta \del u^\epsilon \del u^\mu} \;=\; \frac{\del^3 F}{\del u^\mu \del u^\beta \del u^\gamma} \eta^{\gamma \delta} \frac{\del^3 F}{\del u^\delta \del u^\epsilon \del u^\alpha}$$
such that $\eta^{\gamma \delta}$ is the inverse matrix of the \emph{nondegenerate constant} metric
$$\eta_{\gamma \delta}\;:=\; \frac{\del^3 F}{\del u^1 \del u^\gamma \del u^\delta} \;=\; \mathrm{constant}$$
and such that $F=F(u^1,\ldots,u^n)$ is \emph{quasi-homogeneous}, i.e. there exist a linear \emph{Euler vector field}
$$E = \sum_{\alpha=1}^n d_\alpha u^\alpha \frac{\del}{\del u^\alpha} \;+\; \sum_{\alpha|d_\alpha\neq 0} r^\alpha \frac{\del}{\del u^\alpha} $$
on $U$ for which
$$E(F)\;=\;E^\alpha \frac{\del F}{\del u^\alpha} = d_F F + A_{\alpha \beta} u^\alpha u^\beta + B_\alpha u^\alpha + C$$
where $d_\alpha, r^\alpha, d_F, A_{\alpha \beta}, B_\alpha, C$ are constants.

\vspace{0.5cm}

\begin{rem}
Notice that any linear vector field $E$ of the form $E=(a^\alpha_\beta +r^\alpha)\frac{\del}{\del u^\alpha}$ with $a^\alpha_\beta$ a constant diagonalizable matrix can always be reduced, by a linear change of variables, to the form
$$E = \sum_{\alpha=1}^n d_\alpha u^\alpha \frac{\del}{\del u^\alpha} \;+\; \sum_{\alpha|d_\alpha\neq 0} r^\alpha \frac{\del}{\del u^\alpha} $$
Usually one distinguishes between the case $d_1\neq 0$, for which the normalization $d_1=1$ is chosen, and the case $d_1=0$, for which the Frobenius manifold is called \emph{degenerate}.
\end{rem}

\vspace{0.5cm}

Let $$c_{\alpha \beta \gamma}=\frac{\del^3 F}{\del u^\alpha \del u^\beta \del u^\gamma}.$$ We will always use the flat metric $\eta$ to raise and lower the indices of tensor fields on $U$. In particular we obtain a Frobenius algebra structure on the tangent bundle of $U$ whose structure constants are $c^\alpha_{\beta \gamma}=\eta^{\alpha \mu}\frac{\del^3 F}{\del u^\mu \del u^\beta \del u^\gamma}$. This algebra is associative thanks to WDVV equations and $e:=\frac{\del}{\del u^1}$ is the identity. We will denote by $\star$ the product in such Frobenius algebra. Notice that $\eta_{\alpha\beta}=c_{1\alpha\beta}$.

\vspace{0.5cm}

Two Frobenius structures on $U$ and $U'$ are \emph{equivalent} if there exists a diffeomorphism (or an analytic transformation) $\phi:U\to U'$ such that $\phi$ is a linear conformal transformation of the metrics $\eta$ and $\eta'$
$$\phi^* \eta' = c^2 \eta, \hspace{1cm} c\in\R^* $$
and it induces an isomorphism of Frobenius algebras
$$\phi_*:\;(T_uU,\star) \to (T_{\phi(u)}U',\star ')$$
Notice that this does not necessarily implies $F=\phi^* F'$.

\vspace{0.5cm}

{\bf From Frobenius manifolds to dispersionless BHTS hierarchies.} From a Frobenius manifold one can recover a dispersionless (i.e. $\epsilon=0$) BHTS hierarchy on $\mathcal{L}(\R^n)$ in the following way. Define a new metric (on the cotangent bundle of $U$) as $g_{\alpha \beta}:=E_\gamma c^\gamma_{\alpha \beta}$. The metrics $\eta$ and $g$ define a bihamiltonian structure of hydrodynamic type as
$$\{u^\alpha(x),u^\beta(y)\}_1\;=\; \eta^{\alpha \beta}\delta'(x-y)$$
$$\{u^\alpha(x),u^\beta(y)\}_2\;=\;g^{\alpha \beta}(u) \delta'(x-y)+\Gamma^{\alpha \beta}_\gamma(u) u^\gamma_x \delta(x-y)$$
We define, after Dubrovin, a \emph{deformed flat connection} on $U\times\C^*$ given by
$$\tilde{\nabla}_X Y=X^\alpha \frac{\del}{\del u^\alpha} Y + \zeta X \star Y$$
$$\tilde{\nabla}_{\frac{\del}{\del \zeta}} Y = \frac{\del}{\del \zeta} Y - E \star Y - \frac{1}{\zeta} \mathcal{V} Y$$
where $ X(\zeta),Y(\zeta)\in \Gamma(TU)$, $\zeta\in\C^*$ and $\mathcal{V}^\alpha_\beta=\frac{2-n}{2}\delta^\alpha_\beta + \frac{\del E^\alpha}{\del u^\beta}$. Dubrovin showed (see in particular \cite{D1}, Lecture $2$) that a fundamental solution exists for such connection in the form
$$Y^\alpha_\beta \;=\; \eta^{\alpha \mu}\frac{\del h_\mu(u,\zeta)}{\del u^\beta} \zeta^\mathcal{V} \zeta^R$$
where $R$ is a constant matrix (part of the so called \emph{monodromy data} of the Frobenius manifold, whose explicit form can be recovered explicitly by studying the monodromy at $z=0$ of the the second flatness equation $\del_z Y \;=\; E\star Y + \zeta^{-1} \mathcal{V} Y$) and
$$h_\alpha(u,\zeta)\;=\; \sum_{k=0}^\infty h_{\alpha,k-1}(u)\zeta^k.$$
The functions $h_{\alpha,p}$, for $p=-1,0,1,\ldots$, then form a BHTS hierarchy with respect to the above bihamiltonian structure. In particular we get
$$h_{\alpha,-1}=\eta_{\alpha \beta} u^\beta$$
$$h_{\alpha,0}=\frac{\del F}{\del u^\alpha}$$
$$\Omega_{\alpha,0;\beta,0}=\frac{\del^2 F}{\del u^\alpha \del u^\beta}$$

One can in fact show (see \cite{DZ}) that any dispersionless BHTS hierarchy comes from a Frobenius manifold via this construction (up to a minor technical assumption about semisimplicity of the Frobenius algebras), yealding a full classification of the $\epsilon\;=\;0$ limit.

\vspace{0.5cm}

Frobenius manifolds also provide the link between integrable systems and Gromov-Witten theory, since there is a canonical Frobenius structure on the quantum cohomology of a target symplectic manifold $M$ (assume for simplicity that $H^i(M)=0$ for $i$ odd, in order to avoid grading issues), its Frobenius potential being given by the genus $0$ Gromov-Witten potential without descendants of $M$. In the case of quantum cohomology the Euler vector field is easily determined by the grading conditions on the Gromov-Witten potential, the matrix $R$ is given by the matrix of cup product with the Chern class of $M$, the metric $\eta$ is given by Poincar\'e pairing (the cohomology variables withouth gravitational descendants $t^{\alpha,0}$ playing the role of flat coordinates $u^\alpha$) and the fundamental solution to the flatness equations is provided by
$$h_{\alpha,p}=\left.\frac{\del \mathbf{f}(t)}{\del t^{\alpha,p}}\right|_{t^{\alpha,p}=0,\,p> 0}$$
where $\mathbf{f}(t)$ is the rational Gromov-Witten potential of $M$ with descendants and $t$ globally refers to the cohomology varaibles with descendants $t^{\alpha,p},\, p=0,1,\ldots$.

We will see this in detail in the next sections, using the language of Symplectic Field Theory of $M\times S^1$, which will lead us directly to the hierarchy $h_{\alpha,p}$ without passing explicitly through the Frobenius structure.

\vspace{0.5cm}

\begin{example}
The only one-dimensional Frobenius manifold (up to equivalence) is given by the potential
$$F(u)=\frac{u^3}{6}$$
It corresponds to the quantum cohomology of a point and gives rise to the dispersionless KdV (also called Riemann or Burgers) hierarchy
$$\{u(x),u(y)\}_1=\delta'(x-y)$$
$$\{u(x),u(y)\}_2=u \delta'(x-y)+\frac{1}{2}u_x \delta(x-y)$$
and
$$\mathcal{V}=R=\left(\begin{array}{cc} 0 & 0 \\ 0 & 0\end{array}\right)$$
$$h_p=\frac{u^{p+2}}{(p+2)!}$$
\end{example}

\vspace{0.5cm}

\begin{example}
The Frobenius manifold given by the potential
$$F(u,v)=\frac{u^2v}{2}+e^{v}$$
corresponds to the quantum cohomology of $\mathbb{P}^1$ and gives rise to the dispersionless Toda hierarchy
$$\{u(x),u(y)\}_1=\{v(x),v(y)\}_1=0$$
$$\{u(x),v(y)\}_1=\delta'(x-y)$$
$$\{u(x),u(y)\}_2=e^{v(x)}v_x\delta(x-y)$$
$$\{v(x),v(y)\}_2=2\delta'(x-y)$$
$$\{u(x),v(y)\}_2=u(x)\delta'(x-y)$$
and
$$\mathcal{V}=\left(\begin{array}{cc} -\frac{1}{2} & 0 \\ 0 & \frac{1}{2} \end{array}\right)$$
$$R=\left(\begin{array}{cc} 0 & 0 \\ 2 & 0 \end{array}\right)$$
$$h_1(u,v,\zeta)\;=\; -2e^{zu} \sum_{m\geq 0}\left(\gamma -\frac{1}{2}v+\digamma(m+1)\right)e^{mv}\frac{z^{2m}}{(m!)^2}$$
$$h_2(u,v,\zeta)\;=\; z^{-1} \sum_{m \geq 0} \left(e^{mv+zu}\frac{z^{2m}}{(m!)^2}-1\right)$$
where $\gamma$ is Euler's constant and $\digamma$ is the digamma function.

Notice how the structure of the Frobenius manifold provided us with two infinite sequences of commuting Hamiltonians, whereas the technique of theorem $\ref{recursion}$ only provided one (equivalent, up to a triangular transformation, to $h_{1,p}$). This reflects the regular but resonant nature of Toda hierarchy.
\end{example}

\vspace{0.5cm}

{\bf Solving dispersionless BHTS hierarchies: the hodograph method.} What justifies the name of \emph{integrable system} for a BHTS hierarchy is the possibility of finding a sufficiently complete set of solutions to the equations in a sufficiently explicit form. For dispersionless hierarchies this is performed via the generalized hodograph method (\cite{T},\cite{DZ}), for which we now give a simple but clear proof. Let us start by remarking a very important fact. The Hamiltonians $H_{\alpha,p}[u]=\intt h_{\alpha,p}(u)\di x$ form a complete family of symmetries in the following sense.

\begin{thm}[\cite{DZ}]
For any symmetry $\intt f(u^1,\ldots,u^n) \di x$ of a dispersionless BHTS hierarchy associated to a Frobenius manifold the density $f(u^1,\ldots,u^n)$ satisfies
$$\frac{\del^2 f}{\del u^\alpha u^\beta} = c_{\alpha \beta}^\gamma \frac{\del^2 f}{\del u^\gamma \del u^1}.$$
If, moreover, $f(u^1,\ldots,u^n)$ is polynomial in $u^1$, then there exist $c^{\alpha,p}\in \R$ such that 
$$f(u)=\sum_{\alpha, p} c^{\alpha,p} h_{\alpha,p}(u)$$
where only a finite number of constants $c^{\alpha,p}$ are nonzero.
\end{thm}

Next we notice that any dispersionless BHTS hierarchy is clearly invariant under the space-time rescaling $(x,t)\mapsto(cx,ct)$, where $c\in \R^*$. The infinitesimal generator of such symmetry has the form
$$u^\gamma_c\;=\;\sum_{\alpha,p} u^\gamma_{t^{\alpha, p}} t^{\alpha, p} + u^\gamma_x x \;=\; \left(\sum_{\alpha,p} t^{\alpha,p}\eta^{\gamma \mu}\frac{\del^2 h_{\alpha,p}}{\del u^\mu \del u^\nu}+ x \delta^\gamma_\nu\right) u^\nu_x$$
We can combine such symmetry with any other symmetry $\intt f(u) \di x$ to obtain a new commuting flow
$$u^\alpha_s\;=\; u^\alpha_c + \eta^{\alpha\mu} \frac{\del^2 f}{\del u^\mu \del u^\nu} u^\nu_x$$
and perform stationary reduction of the dynamics at $u^\alpha_s=0$ to get a solution to our BHTS hierarchy in the implicit form
$$\sum_{\alpha,p} t^{\alpha,p}\eta^{\gamma \mu}\frac{\del^2 h_{\alpha,p}}{\del u^\mu \del u^\nu}+ x \delta^\gamma_\nu \;=\; \eta^{\alpha\mu} \frac{\del^2 f}{\del u^\mu \del u^\nu}$$
By the last theorem we conclude
\begin{thm}
Let $\intt f(u) \di x$ be any symmetry of the dispersionless BHTS hierarchy associated to a Frobenius manifold. Then a solution to the hierarchy can be found by solving for $u(x,t)$ the equation
$$x\,\eta_{\gamma 1} \;+\; \sum_{\alpha,p} t^{\alpha,p} \frac{\del h_{\alpha,p-1}(u)}{\del u^\gamma} \;=\; \frac{\del^2 f(u)}{\del u^\gamma \del u^1}$$
If $h(u)$ is polynomial in $u^1$ and $h(u)=\sum_{\alpha, p} c^{\alpha,p} h_{\alpha,p}(u)$, this is equivalent to
$$x\,\eta_{\gamma 1} \;+\; \sum_{\alpha,p} (t^{\alpha,p}-c^{\alpha,p}) \frac{\del h_{\alpha,p-1}(u)}{\del u^\gamma} \;=\;0$$
\end{thm}

\vspace{0.5cm}

\begin{rem}
There is a special solution to a dispersionless BHTS hierarchy, identified by the initial datum $u^\alpha(x,t=0)=\delta^\alpha_1 x$ or, equivalently, by the choice
$$c^{\alpha,0}=\delta^\alpha_1,\hspace{0.5cm} c^{\alpha,p}=0,\ p>0$$
Such solution is called \emph{topological solution} and it is easy to check that its $\tau$-function has the property
$$\epsilon^2 \log \tau_{\mathrm{top}}(t)\Big|\phantom{.}_{\begin{subarray}{l} t^{\alpha,0}=u^\alpha  \\ t^{\alpha,p}=0,\ p>0 \end{subarray}}\;=\; F(u)$$
In the case of Frobenius manifolds coming from quantum cohomology the full topological $\tau$-function coincides with the full genus $0$ Gromov-Witten potential with descendants
$$\mathbf{f}(t) \;=\; \epsilon^2 \log\tau_{\mathrm{top}}(t)$$
This fact is equivalent to validity of topological recursion relations for $\mathbf{f}(t)$, something true in general, and it is the genus $0$ version of Witten's conjecture.
\end{rem}

\vspace{0.5cm}

{\bf Reconstruction of the dispersive tail.} We now briefly address the problem of classifying the possible dispersive deformations ($\epsilon\neq 0$) of a dispersionless BHTS hierarchy. This is important also in view of Dubrovin's reconstruction procedure for higher genera Gromov-Witten invariants starting from a semisimple rational quantum cohomology. In fact semisimplicity of the Frobenius structure plays a central role in the proof of the following main theorem. A Frobenius manifold is called \emph{semisimple} if the Frobenius algebra at a point in each connected component of $U$(and hence on an open, dense subset of $U$) is semisimple. A sufficient condition for semisimplicity is that the operator $(E\star)$ of multiplication by the Euler vector field has pairwise distinct eigenvalues.

\begin{thm}[\cite{DZ}]
All BHTS hierarchies coming from a semisimple Frobnius manifold can be reduced to dispersionless ($\epsilon=0$) form by a \emph{quasi-Miura} transformation
$$u^\alpha \mapsto \tilde{u}^\alpha=u^\alpha+\sum_{k=1}^\infty \epsilon^k F_k^\alpha(u,u_x,u_{xx},\ldots,u^{n_k})$$
where $F_k^\alpha(u,u_x,u_{xx},\ldots,u^{n_k})$ is a rational function in the derivatives of $u$ of degree $k$ with respect to the usual long wave limit grading.
\end{thm}

In other words, the quasi-Miura group acts transitively on the space of deformation of a given semisimple dispersionless BHTS hierarchy, hence providing the tool to  control them.

\vspace{0.5cm}

\begin{example}
The quasi-Miura transformation reducing the KdV hierarchy to the dispersionless KdV hierarchy is given by (see \cite{DZ})
$$\tilde{u}=u -\frac{\epsilon^2}{12}(\log u_x)_{xx} +\epsilon^4\left[\frac{u_{xxxx}}{288 u_x^2}- \frac{7 u_{xx} u_{xxx}}{480 u_x^3} + \frac{u_{xx}^3}{90 u_x^4}\right]+ O(\epsilon^6)$$
\end{example}

\begin{rem}
In the case of semisimple quantum cohomologies, Dubrovin's scheme to reconstruct higher genera invariants starting from genus $0$ consists in finding a specific quasi-Miura transformation for the BHTS hierarchy associated to the relevant semisimple Frobenius manifold and then formulating the following \emph{generalized Witten's conjecture} (\cite{DZ}) about the full Gromov-Witten potential with descendants at all genera
$$\mathbf{F}(t,\hbar)=\epsilon^2 \log \tau_{\mathrm{top}}(t,\epsilon=\sqrt{\hbar})$$
The privileged quasi-Miura transformation is selected using the so called \emph{Virasoro symmetries} (\cite{EHX}). We are not going to discuss them here and refer the reader to \cite{DZ}. Roughly, they are an infinite sequence of extra non-local symmetries of the dispersionless BHTS hierarchy (they can be constructed using just the data of the Frobenius structure) and it is required that the quasi-Miura transformation preserve their expression as linear differential operators in the $t$ variables. This uniquely fixes the transformation, and hence the dispersive deformation of the BHTS hierarchy. Among the properties of such trasformation is the fact that the above topological $\tau$-function depends only on even powers of $\epsilon$, so that $\mathbf{F}(t,\hbar)$ is a power series in the genus expansion variable $\hbar$. In the case of quantum cohomology of a point, Witten's conjecture (\cite{W}) was proved by Kontsevich in \cite{Ko}.
\end{rem}

\vspace{0.5cm}

\section{Gromov-Witten Theory via Symplectic Field Theory}

In this section we construct the dispersionless BHTS hierarchy associated to the quantum cohomology of a symplectic manifold without explicitly passing through Frobenius manifolds, but using instead Symplectic Field Theory. Comparing the formulas in this section with the ones in the previous section it will be evedent how the two integrable systems coincide. Let $(M,\omega)$ be a closed symplectic manifold and, for simplicity, let us assume $H^i(M)=0$ if $i$ is odd.  We will consider Symplectic Field Theory of the framed Hamiltonian structure of fibration type given by the trivial bundle $V=S^1\times M$. Remember that a Hamiltonian structure (see also \cite{E}) is a pair $(V,\Omega)$, where $V$ is an oriented manifold of dimension $2n-1$ and $\Omega$ a closed $2$-form of maximal rank $2n-2$. The line field $\mathrm{Ker} \Omega$ is called the characteristic line field and we will call characteristic any vector field which generates $\mathrm{Ker} \Omega$. A Hamiltonian structure is called stable if and only if there exists a $1$-form $\lambda$ and a characteristic vector field $R$ (called Reeb vector field) such that $$\lambda(R)=1\qquad\mathrm{and}\qquad i_R \di \lambda=0.$$
A framing of a stable Hamiltonian structure is a pair $(\lambda,J)$ with $\lambda$ as above and $J$ a complex structure on the bundle $\xi=\{\lambda=0\}$ compatible with $\Omega$. In particular let $(M,\omega)$ be a closed symplectic manifold with a compatible almost complex structure $J_M$, $p:V\to M$ any $S^1$-bundle and $\lambda$ any $S^1$-connection form; then $(V,\Omega=p^*\omega,\lambda,J)$, with $J$ the lift of $J_M$ to the horizontal distribution, is a framed Hamiltonian structure.\\

In the case of $S^1\times M$, a part of the invariants coming from holomorphic curve counting in $V\times \R$ coincide with those coming from holomorphic curve counting in $M$ and this correspondence will be enough to generate the same integrable system. Moreover the formalism of Symplectic Field Theory is more suitable to appreciate the topological meaning of the algebraic structure hidden in such systems.

\subsection{SFT of Hamiltonian structures of fibration type and Gromov-Witten invariants}

For now let $p:V\to M$ be any $S^1$-bundle over $M$. Recall that, if $J_M$ is an almost complex structure on $M$ compatible with $\omega$, $\lambda$ is any connection $1$-form on V and the complex structure $J$ is any lift of $J_M$ to the horizontal distribution $\xi$, then $(V=S^1\times M, \Omega=p^* \omega,\lambda, J)$ is a framed Hamiltonian structure.

\vspace{0.5cm}

Since each $S^1$-fibre of $V$ is a Reeb orbit for the Reeb vector field and we don't want to destroy the symmetry by perturbing this contact form to a generic one, we will need the non-generic Morse-Bott version of Symplectic Field Theory. Referring to \cite{B} for the general construction, here we just stick to the case of fibrations, where the space $\mathcal{P}$ of periodic Reeb orbits can be presented as $\mathcal{P}=\coprod_{k=1}^\infty \mathcal{P}_k$, where each $\mathcal{P}_k$ is a copy of the base manifold $M$. Let then $\Delta_1,\ldots,\Delta_b$ be a basis of $H^*(M)$ such that the system of forms $\tilde{\Delta}_j:=p^*(\Delta_j)$, $j=1,\ldots,c\le b$ generate $p^*(H^*(M))\subset H^*(V)$, and the forms $\tilde{\Theta}_1,\ldots,\tilde{\Theta}_d$ complete it to a basis of $H^*(V)$. Suppose $H_1(M)=0$ and choose a basis $A_0,A_1,\ldots,A_N$ of $H_2(M)$ in such a way that $\langle c_1(V),A_0 \rangle=l$ (if $l$ is the greatest divisor of the first Chern class $c_1(V)$ of our fibration), $\langle c_1(V),A_i \rangle=0,\ i\neq 0$, and a basis of $H_2(V)$ is given by the lifts of $A_1,\ldots,A_N$ if $l\neq 0$, $A_0,A_1,\ldots,A_N$ if $l=0$.

\vspace{0.5cm}

We consider a graded Poisson algebra $\mathfrak{P}$ formed by series in the formal variables $t^{1,j},\ldots,t^{c,j}$; $\tau^{1,j},\ldots,\tau^{d,j}$ ($j=0,1,\ldots$) associated to the string of forms $(\tilde{\Delta}_1,\ldots,\tilde{\Delta}_c;\tilde{\Theta}_1,\ldots,\tilde{\Theta}_d)$, the variables $p_k^1,\ldots,p_k^b$ associated to the the classes $(\Delta_1,\ldots,\Delta_b)$ and $\hbar$, with coefficients which are polynomials in the variables $q_k^1,\ldots,q_k^b$. The Poisson structure is given in terms of the Poincar\'e pairing $\eta^{ij}$ in $H^*(M)$ as
$$\{p_k^i,q_l^j\}=k\delta_{k,l}\eta^{ij}.$$ Encoding the $p$ and $q$ variables into the the generating funcions $$u^n(x):=\sum_{k=1}^\infty \left(q_k^n\, e^{-\mathrm{i}\,k x}+ p_k^n\, e^{\mathrm{i}\,k x}\right)\qquad n=1,\ldots,b,$$ such Poisson structure can be expressed as a formal distribution $$\{u^i(x),u^j(y)\}=\eta^{ij}\delta'(x-y)$$ where $\eta^{ij}$ is the non-degenerate Poincar\'e pairing in $H^*(M)$.
The grading of the variables is given by:
$$\mathrm{deg}(t^{ij})=2(j-1)+\mathrm{deg}(\tilde{\Delta}_i)$$
$$\mathrm{deg}(\tau^{ij})=2(j-1)+\mathrm{deg}(\tilde{\Theta}_i)$$
$$\mathrm{deg}(q_k^i)=\mathrm{deg} (\Delta_i)-2+2ck$$
$$\mathrm{deg}(p_k^i)=\mathrm{deg} (\Delta_i)-2-2ck$$
$$\mathrm{deg}(z_i)=-2c_1(A_i)$$
where $c=\frac{\langle c_1(TM),A_0\rangle}{l}$ (see \cite{EGH} for details on how to deal with fractional degrees).

\vspace{0.5cm}

The SFT-potential $\mathbf{h}$ is defined as a power series in the $t,\tau,p,q,z$ variables, where the coefficient of the monomial
$$t^{\alpha_1, i_1}\ldots t^{\alpha_r, i_r} \tau^{\beta_1, j_1}\ldots \tau^{\beta_s,j_s} p^{\alpha^+_1}_{k^+_1}\ldots p^{\alpha^+_\mu}_{k^+_\mu} q^{\alpha^-_1}_{k^-_1}\ldots q^{\alpha^-_\nu}_{k^-_\nu} z_1^{D_1}\ldots z_N^{D_N}$$
is given by the integral
\begin{equation*}
\begin{split}
\int_{\overline{\mathcal{M}}^A_{0,r+s,\mu,\nu}}& \mathrm{ev}_1^*\tilde{\Delta}_{\alpha_1}\wedge\psi_1^{i_1}\wedge\ldots \wedge \mathrm{ev}_r^*\tilde{\Delta}_{\alpha_r}\wedge\psi_r^{i_r} \wedge \mathrm{ev}_{r+1}^*\tilde{\Theta}_{\beta_1}\wedge\psi_{r+1}^{j_1}\wedge\ldots \wedge \mathrm{ev}_{r+s}^*\tilde{\Theta}_{\beta_r}\wedge\psi_{r+s}^{j_r} \\
& \wedge (\mathrm{ev}^+_1)^*\Delta^{k_1^+}_{\alpha^+_1} \wedge \ldots \wedge (\mathrm{ev}^+_\mu)^*\Delta^{k_\mu^+}_{\alpha^+_\mu} \wedge (\mathrm{ev}^-_1)^*\Delta^{k_1^-}_{\alpha^-_1} \wedge \ldots \wedge (\mathrm{ev}^-_\nu)^*\Delta^{k_\nu^-}_{\alpha^-_\nu}
\end{split}
\end{equation*}
where $\psi_i=c_1(L_i)$ is the $i$-th psi-class (like in Gromov-Witten theory, but see \cite{F},\cite{FR} for a rigorous definition and further discussion in the SFT setting) and the integral is over the moduli space of holomorphic curves in $V\times \R$ with $r+s$ marked points and $\mu$/$\nu$ positive/negative punctures asymptotically cylindrical over Reeb orbits and realizing, together with the chosen capping surfaces (see \cite{EGH} for details), the homology class $A=\sum D_i A_i$ in $H_2(V)$, modulo the $\R$ action coming from the $\R$ symmetry of the cylindrical target space $V\times \R$. The maps $\mathrm{ev}_i:\overline{\mathcal{M}}^A_{0,r+s,\mu,\nu} \to V$ are the evaluation map at the marked points, while the maps $\mathrm{(ev^\pm_i)}:\overline{\mathcal{M}}^A_{0,r+s,\mu,\nu} \to \mathcal{P}$ are the evaluation maps at the punctures, with $\Delta^{k}_{\alpha}$ the image of $\Delta_\alpha$ in $\mathcal{P}_k$.

Such potential is an element of $\mathfrak{P}$ of degree $-1$ and it satisfies the master equation
$$\{\mathbf{h},\mathbf{h}\}\;=\;0$$

\vspace{0.5cm}

Let us consider the subalgebra $\tilde{\mathfrak{P}}\subset\mathfrak{P}$ obtained by setting to zero all the descendants of the $t$ variables, namely $t^{i,j}=0$ for $j>0$. One can expand $\mathbf{h}|_{t^{i,j}=0,j>0}$ in the number of $\tau$-variables
$$\mathbf{h}|_{t^{i,j}=0,j>0}=\mathbf{h}^0+\sum_{i,j}\mathbf{h}^1_{i,j}\tau^{i,j}+\sum_{i_1,j_1,i_2,j_2}\mathbf{h}^2_{i_1,j_1;i_2,j_2} \tau^{i_1,j_1}\tau^{i_2,j_2}+\ldots$$
and get sequences  $\mathbf{h}^0$, $\mathbf{h}^1_{i,j}$, $\mathbf{h}^2_{i_1,j_1;i_2,j_2}$ of elements in the subalgebra $\tilde{\mathfrak{P}^0}\subset\tilde{\mathfrak{P}}\subset{\mathfrak{P}}$ where all the $s$-variables and the descendants of the $t$-variables are set to zero.

\vspace{0.5cm}

Moreover, expanding in the same way the master equation $\{\mathbf{h},\mathbf{h}\}=0$ one gets that:
\begin{itemize}
\item[1)] $d^0=\{\mathbf{h}^0,\cdot\}:\tilde{\mathfrak{P}}^0\to\tilde{\mathfrak{P}}^0$ makes $\tilde{\mathfrak{P}}^0$ into a differential algebra, since $d^0\circ d^0=0$,
\item[2)] the Poisson bracket on $\tilde{\mathfrak{P}}^0$ descends to the homology $H_*(\tilde{\mathfrak{P}}^0,d^0)$, thanks to the Jacobi identity,
\item[3)] $\{\mathbf{h}^1_{i_1,j_1},\mathbf{h}^1_{i_2,j_2}\}=0$ as homology classes in $H_*(\tilde{\mathfrak{P}}^0,d^0)$.
\end{itemize}
However, because of the $S^1$ symmetry, we will see that that $\mathbf{h}^0=0$, so that $H_*(\tilde{\mathfrak{P}}^0,d^0)=\tilde{\mathfrak{P}}^0$.

\vspace{0.5cm}

The main result of this section is the following
\begin{prop}[\cite{B},\cite{EGH}]\label{bourgeois}
Let $\mathbf{f}_M(\sum t^{i,j}\Delta_i,z)$ be the genus $0$ Gromov-Witten potential of $M$ and $\mathbf{h}_V(\sum t^{i,j}\tilde{\Delta}_i+\sum \tau^{k,l}\tilde{\Theta}_k,q,p)$ the rational SFT potential of $V$ (as a framed Hamiltonian structure of fibration type). Let
$$\mathbf{h}_{k,l}(t,q,p)=\left.\frac{\partial \mathbf{h}_V}{\partial \tau^{k,l}}\left(\sum_{1}^{c} t^{i,j}\tilde{\Delta}_{ij}+\tau^{k,l}\tilde{\Theta}_k,q,p\right)\right|_{\tau^{k,l}=0;\ t^{i,j}=0, j>0}$$
$$h_{k,l}(t;z)=\left.\frac{\partial \mathbf{f}_M}{\partial \tau^{k,l}}\left(\sum_{1}^{b} t^{i,j}\Delta_i+\tau^{k,l}\pi_*\tilde{\Theta}_k,z\right)\right|_{\tau^{k,l}=0;\ t^{i,j}=0, j>0}$$
where $\pi_*$ denotes integration along the fibers of $V$. Then we have
$$\mathbf{h}_{k,l}(t,q,p)=\frac{1}{2\pi}\int_{0}^{2\pi}h_{k,l}(t^{1,0}+u^1(x),\ldots,t^{b,0}+u^b(x),u^{b+1}(x),\ldots,u^c(x);\tilde{z})\di x$$
where $$u^n(x):=\sum_{k=1}^\infty \left(q_k^n\, e^{-\mathrm{i}\,k x}+ p_k^n\, e^{\mathrm{i}\,k x}\right)\qquad n=1,\ldots,b$$
and $\tilde{z}=(\mathrm{e}^{-\mathrm{i}lx} ,z_1,\ldots,z_N)$
\end{prop}
\begin{proof}
The proof is based on a correspondence between the moduli spaces relevant for Gromov-Witten theory of $M$ and Symplectic Field Theory of $V$. Namely, let $p:L\to M$ be the complex line bundle associated to the $S^1$-bundle $V$ and let $\tilde{\mathcal{C}}^{\tilde{d};l_1,\ldots,l_\alpha,m_1,\ldots,m_\beta}_{r+1}(L)$ be the moduli space of pairs $(F=(f;x_1,\ldots,x_{r+1};x_1^+,\ldots,x_\alpha^+;x_1^-,\ldots,x_\beta^-),\sigma)$ where $F$ is a rational holomorphic curve in $M$ from the moduli space $\mathcal{M}^d_{g=0,r+1+\alpha+\beta}(M)$, $\sigma$ is a meromorphic section of the holomorphic line bundle $f^*(L)$ (with first Chern class $d_0l$) such that the divisor of its poles and zeros equals $\sum_{i=1}^\alpha l_i x_i^+ - \sum_{i_1}^\beta m_i x_i^-$ and the degrees are given by $d=(d_0,d_1,\ldots,d_N)$ and $\tilde{d}=(d_1,\ldots,d_N)$. Then each holomorphic curve in $V\times \R$ with cilindrical ends which are asymptotic to periodic Reeb orbits can be identified with an element of $\tilde{\mathcal{C}}^{\tilde{d};l_1,\ldots,l_\alpha,m_1,\ldots,m_\beta}_{r+1}(L)$. Forgetting the section $\sigma$ provides a projection $\mathrm{pr}:\tilde{\mathcal{C}}^{\tilde{d};l_1,\ldots,l_\alpha,m_1,\ldots,m_\beta}_{r+1}(L)/\R \to \mathcal{C}^{\tilde{d};l_1,\ldots,l_\alpha,m_1,\ldots,m_\beta}_{r+1}(L)$, where $\mathcal{C}^{\tilde{d};l_1,\ldots,l_\alpha,m_1,\ldots,m_\beta}_{r+1}(L)$ is the moduli space of holomorphic curves from $\mathcal{M}^d_{g=0,r+1+\alpha+\beta}(M)$ such that $[\sum_{i=1}^\alpha l_i x_i^+ - \sum_{i_1}^\beta m_i x_i^-]=f^*(L)$. Since a meromorphic section of $f^*(L)$ is determined by its divisor of poles and zeros up to a multiplicative complex constant, the fiber of $\mathrm{pr}$ is $S^1$. Actually it is easy to see that this $S^1$ bundle is isomorphic to the pull-back $ev_i^*(V)$ of the $S^1$-bundle $V$ through any of the $r+1$ evaluation maps $ev_i:\mathcal{C}^{\tilde{d};l_1,\ldots,l_\alpha,m_1,\ldots,m_\beta}_{r+1}(L)\to M$. Finally, the result follows by considering that $\mathcal{C}^{\tilde{d};l_1,\ldots,l_\alpha,m_1,\ldots,m_\beta}_{r+1}(L) = \mathcal{M}^d_{g=0,r+1+\alpha+\beta}(M)$ if $\sum l_i -\sum m_j = d_0 l$ and $\mathcal{C}^{\tilde{d};l_1,\ldots,l_\alpha,m_1,\ldots,m_\beta}_{r+1}(L) = \varnothing$ otherwise, since holomorphic line bundles on $\mathbb{P}^1$ are classified by their degree. The theorem, in fact, relates integration over $\tilde{\mathcal{C}}^{\tilde{d};l_1,\ldots,l_\alpha,m_1,\ldots,m_\beta}_{r+1}(L)/\R$ and over $\mathcal{C}^{\tilde{d};l_1,\ldots,l_\alpha,m_1,\ldots,m_\beta}_{r+1}(L)$. Integration over $x$ with the substitution $\tilde{z}=(\mathrm{e}^{-\mathrm{i}lx} ,z_1,\ldots,z_N)$ imposes the condition $\sum l_i -\sum m_j = d_0 l$. Finally notice that the single fiber class $\tilde{\Theta}_k$ serves to fix the remaining $S^1$-symmetry since the Poicaré dual of $\tilde{\Theta}_k$ is transverse to the fibers of $V$.
\end{proof}

Notice that, thanks to the $S^1$-symmetry of the Hamiltonian structure, one can deduce a priori that $\mathbf{h}^0=0$, so that $H_*(\tilde{\mathfrak{P}},d^0)=\tilde{\mathfrak{P}}$.

\begin{rem}
The above result can be easily generalized (see \cite{R1}) to the case where $V$ is the total space of an $S^1$-orbibundle over a closed symplectic orbifold. In this case everything is to be rephrased in terms of orbifold cohomology and orbifold Gromov-Witten theory in the sense of \cite{CR},\cite{CR1}.
\end{rem}

\subsection{The trivial bundle case: Gromov-Witten theory}

Let us now get back to the trivial bundle case $V=S^1\times M$. Choose $\tilde{\Theta}_i=\tilde{\Delta}_i\wedge \di \phi$, $i=1,\ldots,b$, where $\phi$ is the angle variable on $S^1$ and $b$ denotes the dimension of $H^*(M)$. Moreover let $A_1,\ldots,A_N$ be a basis in $H_2(M)$. In this case the Hamiltonians appearing in Theorem \ref{bourgeois}, in terms of the Gromov-Witten potential of $M$, are given by $$\mathbf{h}_{k,l}(t,q,p)=\frac{1}{2\pi}\int_{0}^{2\pi}h_{k,l}(v^1(x),\ldots,v^b(x))\di x$$ with $v^i(x)=t^i+u^i(x)$ and $$h_{k,l}=\left.\frac{\partial \mathbf{f}_M}{\partial t^{k,l}}\right|_{t^{i,j}=0, j>0}$$

Recall that, by definition of Gromov-Witten potential, $\eta_{ij}=\partial_1\partial_i\partial_j \mathbf{f}_M$ and denote $t^i=t^{i,0}$, $\mathbf{h}_i=\mathbf{h}_{i,0}$ and $h_i(v)=h_{i,0}(v)$. Denote moreover $c_{ijk}=\partial_i\partial_j \partial_k \mathbf{f}_M$. We now show that $c_{ij}^k$ (where we use the metric $\eta$ to raise and lower the indices) are structure constants of a commutative associative algebra on each tangent space $T_tH^*(M)$. Notice preliminarly that, if $F(t,p,q) = \frac{1}{2\pi}\int_0^{2\pi}f(v(x))\di x$ and $G(t,p,q) = \frac{1}{2\pi}\int_0^{2\pi}g(v(x))\di x$, then $$\{F,G\}=\frac{1}{2\pi}\int_0^{2\pi}\partial_i f \eta^{ij}\partial_x \partial_j g\ \di x$$ hence
$$\{F,G\}=0 \ \Leftrightarrow \ \exists\ \Omega=\Omega(v)\ | \ \partial_i f \eta^{ij}\partial_x \partial_j g=\partial_x \Omega$$
This is in turn equivalent to exactness of the $1$-form $\partial_i f \eta^{ij}\partial_x \partial_j g \di v^k$ and hence amounts to 
\begin{equation}\label{poissoncommute}
\partial_j\partial_i f \partial_k\partial^i g=\partial_k\partial_i f\partial_j\partial^i g
\end{equation}
Commutativity follows trivially by definition of $c_{ij}^k$. Associativity is equivalent to the Gromov-Witten potential $\mathbf{f}_M(t)$ of $M$ satisfying WDVV associativity equations
\begin{equation}\label{wdvv}
\partial_\alpha\partial_k\partial^i \mathbf{f}_M \partial_i\partial_\beta\partial_j \mathbf{f}_M=\partial_k\partial_\beta\partial^i \mathbf{f}_M \partial_\alpha\partial_i\partial_j \mathbf{f}_M
\end{equation}
This is easily verified using the equations $\{\mathbf{h}_\alpha,\mathbf{h}_\beta\}=0$ and (\ref{poissoncommute}).

\vspace{0.5cm}

Exactly in the same way, but now considering the equations $\{\mathbf{h}_\alpha,\mathbf{h}_{\beta, j}\}=0$ we get $$c_{\alpha i}^k\partial_k\partial_s h_{\beta, j} = c_{\alpha s}^k\partial_i\partial_k h_{\beta, j}$$ and, for $i=1$,
\begin{equation}\label{commutingalgebra}
\partial_\alpha\partial_s h_{\beta, j}-c_{\alpha s}^k\partial_k\partial_1 h_{\beta, j}=0
\end{equation}

These last equations, together with WDVV equations (\ref{wdvv}), are actually equivalent to commutativity of the whole system of Hamiltonians $\mathbf{h}_{\alpha, j}$ as the follwing simple Lemma shows.

\begin{lemma}
\begin{equation*}
\{\mathbf{h}_{\alpha, i},\mathbf{h}_{\beta, j}\}=0 \Leftrightarrow \left\{\begin{array}{l} \{\mathbf{h}_\alpha,\mathbf{h}_\beta\}=0\\ \partial_\alpha\partial_s h_{\beta, j}-c_{\alpha s}^k\partial_k\partial_1 h_{\beta, j}=0\end{array}\right.
\end{equation*}
\end{lemma}
\begin{proof}
We only need to prove the $\Leftarrow$ part. Using $\partial_\alpha\partial_s h_{\beta j}-c_{\alpha s}^k\partial_k\partial_1 h_{\beta j}=0$ we get
$$\partial_\nu \partial_\gamma h_{\alpha i}\partial^\gamma \partial_\delta h_{\beta j}-\partial_\nu \partial_\gamma h_{\beta j}\partial^\gamma \partial_\delta h_{\alpha i}=c_{\nu \gamma}^\epsilon {c^\gamma_\delta}^\mu(\partial_1\partial_\epsilon h_{\alpha i}\partial_1 \partial_\mu h_{\beta j}- \partial_1\partial_\mu h_{\alpha i}\partial_1 \partial_\epsilon h_{\beta j})$$ which vanishes thanks to associativity. This is equivalent to $\{\mathbf{h}_{\alpha, i},\mathbf{h}_{\beta, j}\}=0$.
\end{proof}

Finally we need to put into play the information about the grading of our generating functions. In particular we know that, with the above choice of degrees for the involved variables, the Gromov-Witten potential has degree $2(n-3)$ if $2n$ is the dimension of $M$. This can be expressed by saying that there exists an Euler vector field, linear in the variables $t$ and $z$, of the form:
$$\tilde{E}=\sum_{\alpha=1}^b\left(-2+\mathrm{deg}\, \Delta_\alpha\right)t^\alpha \frac{\partial}{\partial t^\alpha}- \sum_{i=1}^N 2c_1(A_i) z_i \frac{\partial}{\partial z_i}$$
and, recalling the definition of the Hamiltonian densities $h_{\alpha, j}$, such that
\begin{equation}\label{homogeneity}
\tilde{E}(h_{\alpha, j})=\{[2(n-3)]-[2(j-1)+\mathrm{deg}\, \Delta_\alpha]\}\, h_{\alpha, j}
\end{equation}

\vspace{1cm}

In order to proceed further we will need some properties and recursion relations for Gromov-Witten invariants. We will state them here, without proving them, in terms of the Gromov Witten correlators with descendants
$$\langle \Delta_{i_1}\psi^{j_1},\ldots,\Delta_{i_n}\psi^{j_n} \rangle_g^A=\int_{\bar{\mathcal{M}}_{g,n}^A(M)}\mathrm{ev}_1^*(\Delta_{i_1})\wedge \psi_1^{j_1}\wedge\ldots\wedge\mathrm{ev}_n^*(\Delta_{i_n})\wedge \psi_n^{j_n}$$
where, as usual, $g$ is the genus, $n$ the number of marked points and $A\in H_2(M)$ is the degree of the curves in the moduli space $\bar{\mathcal{M}}_{g,n}^A(M)$. We address the reader to \cite{CK} for a clear discussion and proof.

\vspace{0.5cm}

The first one is called \emph{fundamental class axiom} and it takes the form 
$$\langle \Delta_{i_1}\psi^{j_1},\ldots,\Delta_{i_{n-1}}\psi^{j_{n-1}},1 \rangle_g^A=\sum_{k=1}^{n-1} \langle \Delta_{i_1}\psi^{j_1},\ldots,\Delta_{i_{k-1}}\psi^{j_{k-1}},\Delta_{i_k}\psi^{j_k-1},\ldots, \Delta_{i_{n-1}}\psi^{j_{n-1}}\rangle_g^A$$
The direct consequence we need is easily translated into the following relation among Hamiltonians
$$\partial_1 h_{\alpha,j}=h_{\alpha, j-1}$$
This equation can be plugged into equation (\ref{commutingalgebra}) to get
\begin{equation}\label{commutativityfundclass}
\partial_\alpha\partial_s h_{\beta,j}-c_{\alpha s}^k\partial_k h_{\beta,j-1}=0
\end{equation}

\vspace{0.5cm}

The second axiom we need is the \emph{divisor axiom} and, if $D\in H^2(M)$ and either $D\neq 0$ or $n\geq 4$, it spells out as
\begin{equation}
\begin{split}
\langle  \Delta_{i_1}\psi^{j_1},\ldots,& \Delta_{i_{n-1}}\psi^{j_{n-1}},D \rangle_g^A= (\int_A D) \langle \Delta_{i_1}\psi^{j_1},\ldots, \Delta_{_{n-1}} \psi^{j_{n-1}}\rangle_g^A\\
&+\sum_{k=1}^{n-1}\langle \Delta_{i_1}\psi^{j_1},\ldots,\Delta_{i_{k-1}}\psi^{j_{k-1}},(D\cup\Delta_{i_k})\psi^{j_k-1},\ldots, \Delta_{i_{n-1}}\psi^{j_{n-1}}\rangle_g^A
\end{split}
\end{equation}
First of all notice that this means our generating functions depend polynomially on all the $t$ variables (by degree considerations) but those associated with $2$-forms (which have degree $0$). If $D=\sum \bar{t}_i \Delta_i$ denotes the generic $2$-form, then $h_{\alpha,j}$ depends polynomially on $z^A\mathrm{e}^{(D,A)}$ up to a further polynomial dependence on the $\bar{t}_i$'s, corresponding to the sum in the r.h.s of the divisor axiom, which involves the classical cup product and lower descendants. Notice that the $\bar{t}_i$'s also appear polynomially in the term accounting for degree $0$ curves with three marked points, which has the form $\int_M t^{\wedge 3}$. These considerations can be promptly translated in terms of generating functions by defining an Euler vector field on $H^*(M)$ (i.e. one which doesn't involve the $z$ variables anymore) of the form
$$E=\sum_{\alpha=1}^b\left[\left(-2+\mathrm{deg}\, \Delta_\alpha\right)t^\alpha+ r^\alpha\right]\frac{\partial}{\partial t^\alpha}$$
where $r^\alpha=-2c_1(\Delta_\alpha)$. This Euler vector field assigns degree $-2c_1(\Delta_\alpha)$ to $\mathrm{e}^{t_\alpha}$ when $\Delta_\alpha \in H^2(M)$ and, in light of the above considerations, we can rewrite the homogeneity condition (\ref{homogeneity}) as
\begin{equation}\label{gradingdivisor}
E(h_{\alpha,j}) -r^i\, {{{c_i}^k}_{\alpha}|}_{t=0}\, h_{k,j-1} = \{[2(n-3)]-[2(j-1)+\mathrm{deg}\, \Delta_\alpha]\}\,  h_{\alpha, j}
\end{equation}
where ${{{c_i}^k}_{\alpha}|}_{t=0}$ are the structure constants of the classical cup product in $H^*(M)$ and, of course, the term which involves them is the correction term needed to take into account the further polynomial dependence of $h_{\alpha,j}$ on $t_\alpha$ when $\Delta_\alpha \in H^2(M)$ (both in degree $0$ and $n=3$, using $h_{\alpha,-1}=\sum \eta_{\alpha\beta} t_\beta$, and in the other cases as discribed by the divisor axiom).

\vspace{0.5cm}

It can be shown (see \cite{DZ},\cite{D1}) that equations (\ref{commutativityfundclass}) and (\ref{gradingdivisor}) completely determine the hamiltonians $h_{\alpha,j}$ by recursion, starting from $h_{\alpha,-1}=\sum \eta_{\alpha\beta} t_\beta$. In fact, if we define the generating functions of the hamiltonian densities as
$$h_\alpha(t,\zeta)=\sum_{k=0}^\infty h_{\alpha,k-1}(t) \zeta^{k}=\sum_A\langle\frac{\Delta_\alpha}{1-\zeta \psi},t,\ldots,t \rangle_0^A z^A$$
and form the so called $J$-function
$$J_\alpha(t,\zeta)=h_\alpha(t,\zeta)\,\zeta^\mathcal{V}\, \zeta^{R}$$
where the matrix $\mathcal{V}$ is given by $\mathcal{V}_\epsilon^\gamma=\frac{2-n}{2}\delta_\epsilon^\gamma+\frac{1}{2}\partial_\epsilon E^\gamma$ and $R$ is the matrix of (classical) multiplication by the first Chern class $R_\epsilon^\gamma=-\frac{1}{2}r^i\, {{{c_i}^\gamma}_{\epsilon}|}_{t=0}$, then the equations above are equivalent to $J_\alpha$ being flat coordinates for the already mentioned \emph{deformed flat connection} on $H^*(M)\times\C^*$
$$\tilde{\nabla}_X Y=X^\alpha \partial_\alpha Y + \zeta X \star Y$$
$$\tilde{\nabla}_{\frac{\di}{\di \zeta}} Y = \frac{\di}{\di \zeta} Y -\frac{1}{2} E \star Y - \frac{1}{\zeta} \mathcal{V} Y$$
where $ X(\zeta),Y(\zeta)\in TH^*(M)$ and $\zeta\in\C^*$. Notice that, beacause of the grading definition in Symplectic Field Theory \cite{EGH}, our Euler vector Field $E$ (and hence also the operator $R$) differs by a $-\frac{1}{2}$ multiplicative factor from Dubrovin's one \cite{DZ}, which we defined above and which is more commonly used in the literature.

\end{document}